\documentclass[reqno, 12pt]{amsart}   
\usepackage{amsmath,amsthm,amscd,amssymb,graphicx,enumerate,latexsym}

\usepackage[raiselinks,colorlinks]{hyperref}
\hypersetup{citecolor=blue}

\numberwithin{equation}{section}

\theoremstyle{plain}
\newtheorem{thm}{Theorem}[section]
\newtheorem{lemma}{Lemma}[section]
\newtheorem{cor}[thm]{Corollary} 
\theoremstyle{definition}
\newtheorem{defn}{Definition}[section]
\theoremstyle{remark}
\newtheorem{remark}{Remark}[section]

\def\Re{\mathop{\rm Re}\nolimits}

\newcommand{\dist}{\text{\rm{dist}}}

\newcommand{\ac}{\text{\rm{ac}}}
\newcommand{\singc}{\text{\rm{sc}}}

\newcommand{\pp}{\text{\rm{pp}}}

\newcommand{\bi}{\bibitem}

\newcommand{\beq}{\begin{equation}}
\newcommand{\eeq}{\end{equation}}
\newcommand{\ba}{\begin{align}}
\newcommand{\ea}{\end{align}}

\DeclareMathOperator*{\esssupp}{ess\,supp}

\DeclareMathOperator*{\supp}{supp}
\allowdisplaybreaks

\newcommand{\rel}[1]{\sim_{#1}}

\renewcommand{\MRhref}[2]{\href{http://www.ams.org/mathscinet-getitem?mr=#1}{#2}}
\renewcommand{\MR}[1]{}

\makeatletter
\def\@strippedMR{}
\def\@scanforMR#1#2#3\endscan{%
   \ifx#1M\ifx#2R\def\@strippedMR{#3}%
   \else\def\@strippedMR{#1#2#3}%
   \fi\fi}
\renewcommand\MR[1]{\relax\ifhmode\unskip\spacefactor3000 \space\fi
   \@scanforMR#1\endscan
   \MRhref{\@strippedMR}{MR\@strippedMR}}
\makeatother

\newcommand{\GBV}[1]{GBV(#1)}

\title[generalized bounded variation]{Orthogonal polynomials with recursion coefficients of generalized bounded variation}
\author{Milivoje Lukic}
\date\today
\thanks{Mathematics 253-37, California Institute of Technology, Pasadena, CA 91125, USA.
E-mail: mlukic@caltech.edu}

\keywords{Orthogonal polynomial, bounded variation, Wigner--von Neumann potential}
\subjclass[2000]{42C05,47B36}    

\begin{document}

\begin{abstract}
We consider probability measures on the real line or unit circle with Jacobi or Verblunsky coefficients satisfying an $\ell^p$ condition and a generalized bounded variation condition. This latter condition requires that a sequence can be expressed as a sum of sequences $\beta^{(l)}$, each of which has rotated bounded variation, i.e.,
\begin{equation*}
\sum_{n=0}^\infty \lvert e^{i\phi_l} \beta_{n+1}^{(l)} - \beta_n^{(l)} \rvert < \infty
\end{equation*}
for some $\phi_l$. This includes discrete Schr\"odinger operators on a half-line or line 
with finite linear combinations of Wigner--von Neumann type potentials.

For the real line, we prove that in the Lebesgue decomposition $d\mu = f dm + d\mu_s$ of such measures, $\operatorname{supp}(d\mu_s) \cap (-2,2)$ is contained in an explicit finite set $S$ (thus, $d\mu$ has no singular continuous part), and $f$ is continuous and non-vanishing on $(-2,2) \setminus S$. The results for the unit circle are analogous, with $(-2,2)$ replaced by the unit circle.
\end{abstract}

\maketitle

\section{Introduction}
In this paper we will be interested in orthogonal polynomials on the unit circle (OPUC) and orthogonal polynomials on the real line (OPRL). We will state the necessary definitions, but for more information on  OPUC and OPRL, we refer the reader to \cite{Sze,GBk,FrB,Chi,OPUC1,OPUC2,Rice}.

To each probability measure on the unit circle $d\mu(\theta) = w(\theta)\frac{d\theta}{2\pi} + d\mu_s$ of infinite support, there corresponds a sequence of orthonormal polynomials $\varphi_n(z)$ with $\deg\varphi_n = n$ and $\int \bar \varphi_m(z) \varphi_n(z) d\mu = \delta_{mn}$ obeying the Szeg\H o recursion relation
\begin{equation}\label{1.1}
z \varphi_n(z) = \sqrt{1-\lvert \alpha_n\rvert^2} \, \varphi_{n+1}(z) + \bar \alpha_n \varphi_n^*(z)
\end{equation}
with $\varphi_n^*(z) = z^n \overline{ \varphi_n (1/\bar z)}$ and with $\alpha_n\in \mathbb{D}=\bigl\{z\in \mathbb{C} \big\vert \lvert z\rvert <1\bigr\}$ called Verblunsky coefficients. By a theorem of Verblunsky \cite{Ver}, this is a bijective correspondence between such measures and sequences $\{\alpha_n\}_{n=0}^\infty$ with $\alpha_n \in \mathbb{D}$.

To each probability measure on the real line $d\rho(x) = f(x) dx + d\rho_s(x)$ of infinite but bounded support, there corresponds a sequence of orthonormal polynomials $p_n(x)$ with $\deg p_n = n$ and $\int p_m(x) p_n(x) d\rho = \delta_{m n}$ obeying the Jacobi recursion relation
\begin{equation}\label{1.2}
x p_n(x) = a_{n+1} p_{n+1}(x) + b_{n+1} p_n(x) + a_n p_{n-1}(x)
\end{equation}
with $a_n>0$, $b_n\in\mathbb{R}$ called Jacobi coefficients. By a theorem of Stieltjes \cite{Sti}, more commonly known as Favard's theorem, 
this is a bijective correspondence between such measures and sequences $\{a_n,b_n\}_{n=1}^\infty$  with $a_n>0$, $b_n\in \mathbb{R}$, and
\[
\sup_n a_n  + \sup_n \lvert b_n \rvert < \infty
\]

Next we discuss the generalized bounded variation condition.

\begin{defn}\label{D1.1} A sequence $\beta=\{\beta_n\}_{n=N}^\infty$ ($N$ can be finite or $-\infty$) has \emph{rotated bounded variation} with phase $\phi$ if
\begin{equation}\label{1.3}
\sum_{n=N}^\infty \lvert  e^{i\phi} \beta_{n+1} - \beta_n \rvert < \infty
\end{equation}
A sequence $\alpha = \{\alpha_n\}_{n=N}^\infty$ has \emph{generalized bounded variation} with the set of phases $A=\{\phi_1,\dots,\phi_L\}$ if it can be expressed as a sum
\begin{equation}\label{1.4}
\alpha_n = \sum_{l=1}^L \beta_n^{(l)}
\end{equation}
 of $L<\infty$ sequences $\beta^{(1)},\dotsc,\beta^{(L)}$, such that the $l$-th sequence $\beta^{(l)}$ has rotated bounded variation with phase $\phi_l$. The set of sequences having generalized bounded variation with set of phases $A$ will be denoted $\GBV{A}$ or, with a slight abuse of notation, $\GBV{\phi_1,\dots,\phi_L}$. In particular, $\GBV{\phi}$ is the set of sequences with rotated bounded variation with phase $\phi$.
\end{defn}

For an example of  rotated bounded variation with phase $\phi$, take $\beta_n = e^{-i(n\phi+\alpha)} \gamma_n$, with $\{\gamma_n\}_{n=N}^\infty$ any sequence of bounded variation. Generalized bounded variation may seem like an unnatural condition for real-valued sequences, but by combining rotated bounded variation with phases $\phi$ and $-\phi$, one gets
\[
e^{-i(n\phi+\alpha)} \gamma_n + e^{+i(n\phi+\alpha)} \gamma_n  = \cos(n\phi+\alpha) \gamma_n
\]
It is then clear that a linear combination of Wigner--von Neumann type potentials plus an $\ell^1$ part,
\begin{equation}\label{1.5}
V_n = \sum_{k=1}^K \lambda_k \cos(n\phi_k+\alpha_k)/ n^{\gamma_k}  + W_n 
\end{equation}
with $\gamma_k>0$ and $\{W_n\}\in\ell^1$, has generalized bounded variation.

We can now state the two central results of this paper.

\begin{thm}[OPUC]\label{T1.1}
Let $d\mu = w(\theta)\frac{d\theta}{2\pi}+d\mu_s$ be a probability measure on the unit circle with infinite support and $\{\alpha_n\}_{n=0}^\infty$ its Verblunsky coefficients. Assume that 
\[
\{\alpha_n\}_{n=0}^\infty\in\ell^p \cap \GBV{A}
\]
for a positive odd integer $p=2q+1$ and a finite set $A\subset \mathbb{R}$. Let
\begin{equation}\label{1.6}
S = \bigl\{\exp(i \eta) \big\vert \eta \in (\underbrace{A+\dots+A}_{q\text{ times}}) - (\underbrace{A+\dots+A}_{q-1\text{ times}})\bigr\}
\end{equation}
Then
\begin{enumerate}[\rm (i)]
\item\label{T1.1(i)} $\supp \mu_s \subset S$  and, in particular, $d\mu$ has no singular continuous part;
\item\label{T1.1(ii)} $w(\theta)$ is continuous and strictly positive on $\partial\mathbb{D} \setminus S$.
\end{enumerate}
\end{thm}

\begin{thm}[OPRL] \label{T1.2} Let $d\rho = f(x) dx +d\rho_s$ be a probability measure on the real line with infinite support and finite moments and $\{a_n, b_n\}_{n=1}^\infty$ its Jacobi coefficients. Let $p$ be a positive integer, $A\subset \mathbb{R}$ a finite set of phases, and make one of these sets of assumptions:
\begin{enumerate}[\rm $1^\circ$]
\item \label{T1.2(1circ)} $\{a_n^2-1\}_{n=1}^\infty, \{b_n\}_{n=1}^\infty \in \ell^p \cap \GBV{A}$
\item  \label{T1.2(2circ)} $\{a_n-1\}_{n=1}^\infty, \{b_n\}_{n=1}^\infty \in \ell^p \cap \GBV{A}$
\end{enumerate}
Denote
$\tilde A = A \cup \{0\}$ in case {\rm 1$^\circ$} and $\tilde A = (A+A)\cup A \cup \{ 0\}$ in case {\rm 2$^\circ$}, and let
\begin{equation}\label{1.7}
S = \bigl\{ 2 \cos(\eta/2) \big\vert \eta \in  \underbrace{\tilde A+\dots+\tilde A}_{p-1\text{ times}} \bigr\}
\end{equation}
Then
\begin{enumerate}[\rm (i)]
\item\label{T1.2(i)} $\supp  \rho_s \cap (-2,2) \subset S$ and, in particular, $d\rho$ has no singular continuous part;
\item\label{T1.2(ii)} $f(x)$ is continuous and strictly positive on $(-2,2) \setminus S$.
\end{enumerate}
\end{thm}

\begin{remark}
As we will see later, since recursion coefficients are in $\ell^p$, all their constituent sequences of rotated bounded variation are in $\ell^p$. However, if some of these constituent sequences have faster decay, this can be used to reduce the set $S$. Namely, a phase $\phi_1+\dots+\phi_k - \phi_{k+1} - \dots - \phi_{k+l}$ must only be included in \eqref{1.6} or \eqref{1.7} if the pointwise product of the corresponding sequences, $\{ \beta_n^{(1)} \cdots \beta_n^{(k)} \bar \beta_n^{(k+1)} \cdots \bar \beta_n^{(k+l)} \}$, is not in $\ell^1$. The proofs of Theorems~\ref{T1.1} and \ref{T1.2} in this paper can be easily modified to show this.
\end{remark}

\begin{remark}
By Lemma~\ref{L2.2}(\ref{L2.2(vi)}) shown later in this paper,
\[
\{a_n-1\}_{n=1}^\infty \in \GBV{A} \implies \{a_n^2-1\}_{n=1}^\infty \in \GBV{(A+A)\cup A} 
\]
Also, $\{a_n-1\}_{n=1}^\infty \in \ell^p$ implies $\{a_n^2-1\}_{n=1}^\infty \in \ell^p$. Thus, with the replacement of the set $A$ by $(A+A)\cup A$, case \ref{T1.2(1circ)}$^\circ$ of Theorem~\ref{T1.2} implies case \ref{T1.2(2circ)}$^\circ$. For that reason, in the remainder of the paper we will only discuss case \ref{T1.2(1circ)}$^\circ$ of Theorem~\ref{T1.2}.
\end{remark}

Theorem~\ref{T1.2} can be viewed in the special case $a_n=1$, where it becomes a result on discrete Schr\"odinger operators on a half-line. Using a standard pasting argument, this also implies a result for discrete Schr\"odinger operators on a line.

\begin{cor}[1D discrete Schr\"odinger operators] \label{C1.3} Let
\begin{equation}\label{1.8}
(Hx)_n = x_{n+1} + x_{n-1} + V_n x_n
\end{equation}
 be a discrete Schr\" odinger operator on a half-line or line, with $\{V_n\}$ in $\ell^p$ with generalized bounded variation with set of phases $A$. Then
\begin{enumerate}[\rm (i)]
\item $\sigma_\ac(H) = [-2,2]$
\item $\sigma_\singc (H) = \emptyset$
\item $\sigma_\pp (H) \cap (-2,2)$ is a finite set,
\[\sigma_\pp (H) \cap (-2,2)\subset \Bigl\{ 2 \cos(\eta/2) \Big\vert \eta \in \bigcup_{k=1}^{p-1} (\underbrace{A+\dots+A}_{k\text{ times}}) \Bigr\}
\]
\end{enumerate}
\end{cor}

This corollary applies in particular to linear combinations of Wigner--von Neumann potentials \eqref{1.5}.

Spectral consequences of bounded variation coupled with convergence of recursion coefficients are well known. These results are often cited as Weidmann's theorem, who proved the first result of this kind, for Schr\"odinger operators \cite{Wei}. The analogous OPRL result, due to M\'at\'e--Nevai \cite{MN}, states that bounded variation of $\{a_n\}_{n=1}^\infty$ and $\{b_n\}_{n=1}^\infty$ together with  $a_n \to 1$, $b_n \to 0$ implies Theorem~\ref{T1.2}(\ref{T1.2(i)}),(\ref{T1.2(ii)}) with $S=\emptyset$. The corresponding result for OPUC, by Peherstorfer--Steinbauer \cite{PS}, states that bounded variation of $\{\alpha_n\}_{n=0}^\infty$ together with $\alpha_n \to 0$ implies Theorem~\ref{T1.1}(\ref{T1.1(i)}),(\ref{T1.1(ii)}) with $S=\{1\}$. Rotating the measure on the unit circle gives an immediate corollary, that rotated bounded variation of $\{\alpha_n\}_{n=0}^\infty$ with phase $\phi$ together with  $\alpha_n \to 0$ implies Theorem~\ref{T1.1}(\ref{T1.1(i)}),(\ref{T1.1(ii)}) with $S=\{e^{i\phi}\}$.  Wong \cite{Won} has the first result to consider multiple phases, proving  Theorem~\ref{T1.1} in the case $\{\alpha_n\}_{n=0}^\infty \in \ell^2$.  During the writing of this paper, we learned about work by Janas--Simonov \cite{JS} analyzing potentials of the form $V_n = \cos(\phi n + \delta) / n^\gamma + q_n$, with $\gamma>1/3$ and $\{q_n\}_{n=1}^\infty \in \ell^1$. They obtain the same spectral results as our Corollary~\ref{C1.3} by a different method.


As communicated to us by Yoram Last, this problem can also be motivated in a different way:
let $V_n = W_n f_n$, with $f_n>0$ monotone decaying to $0$, and let $H$ be given by \eqref{1.8}. For different classes of potentials $\{W_n\}$, what kind of decay do we need to ensure $\sigma_\singc(H) = \emptyset$?  If $\{W_n\}$ is periodic, the method of Golinskii--Nevai \cite{GN} shows that any such $\{f_n\}$ will suffice. If $\{W_n\}$ are i.i.d. random variables, Kiselev--Last--Simon \cite{KLS} have shown that $\{f_n\} \in \ell^2$ is needed. By our Corollary~\ref{C1.3}, if $\{W_n\}$ is the almost periodic potential
$
W_n = \lambda \cos(n\phi+\alpha)
$,
then any $\{f_n\} \in \ell^p$ (with any $p<\infty$) will suffice.

The remainder of this paper is dedicated to proofs of Theorems~\ref{T1.1} and \ref{T1.2}. In Section 2, we discuss some properties of sequences of generalized bounded variation. In Sections 3--5, we introduce Pr\"ufer variables for OPUC and OPRL and present them in a unified way which will enable us to present a shared proof of the two theorems. In Sections 6 and 7 we present proofs in the $\ell^2$ and $\ell^3$ cases, building up the tools for the general proof in Sections 8 and 9.

I would like to thank my advisor, Professor Barry Simon, for suggesting this problem and for his guidance and helpful discussions.

\section{Generalized Bounded Variation}

In this section we describe some properties of sequences of rotated and generalized bounded variation. Most importantly, we prove that if a sequence is of generalized bounded variation and is in some $\ell^p$ space, then all the constituent sequences are also in $\ell^p$. 

\begin{lemma}\label{L2.1}
Let $\alpha \in \GBV{\phi_1,\dotsc,\phi_L}$, with decomposition \eqref{1.4} into sequences of rotated bounded variation. Then for any $1\le p\le \infty$,
\[
\alpha \in \ell^p \implies \beta^{(1)},\dotsc,\beta^{(L)} \in\ell^p
\]
\end{lemma}

\begin{proof} We will prove $\beta^{(1)} \in \ell^p$; the proof for any $\beta^{(l)}$ is analogous.
Let $T$ be the shift operator on sequences, defined by $Tz = \{ z_{n+1} \}_{n=N}^\infty$ for $z=\{z_n\}_{n=N}^\infty$. In terms of $T$, the condition \eqref{1.3} can be rewritten as
\begin{equation}\label{2.1}
(e^{i\phi_l} T - 1) \beta^{(l)} \in \ell^1
\end{equation}
Note that for any $1\le q\le \infty$, $z\in \ell^q$ implies $Tz \in \ell^q$; thus, for an arbitrary polynomial $P(T)$,
\begin{equation}\label{2.2}
z \in \ell^q  \implies  P(T) z \in \ell^q
\end{equation}
Now let $Q(T) = \prod_{l=2}^L (e^{i\phi_l} T - 1)$. By \eqref{2.2} with $q=1$, \eqref{2.1} implies $Q(T)\beta^{(l)} \in \ell^1$ for $l\neq 1$. Meanwhile, $\alpha\in\ell^p$ and  \eqref{2.2} imply $Q(T)\alpha \in \ell^p$. Thus, applying $Q(T)$ to \eqref{1.4} gives
\begin{equation}\label{2.3}
Q(T) \beta^{(1)} = Q(T) \alpha - \sum_{l=2}^L Q(T) \beta^{(l)} \; \in \ell^p
\end{equation}
Since the $\phi_l$ are mutually distinct, $Q(T)$ is coprime with $e^{i\phi_1} T - 1$, so there exist complex polynomials $U(T), V(T)$ such that
\begin{equation*}
1 = U(T) Q(T) + V(T) (e^{i\phi_1} T - 1)
\end{equation*}
Thus, applying $U(T)$ to \eqref{2.3} and $V(T)$ to $(e^{i\phi_1} T - 1) \beta^{(1)} \in \ell^1$ and adding the two, we obtain $\beta^{(1)} \in \ell^p$. 
\end{proof}

\begin{remark}\label{R2.1}
If a sequence $\alpha$ is of generalized bounded variation, uni\-queness of the representation \eqref{1.4} is of some interest. Clearly, we can freely add $\ell^1$ sequences to $\beta^{(l)}$'s, as long as the sum of those sequences cancels out in $\alpha$. By doing so, we can eliminate any extraneous $\beta^{(l)}$ which are in $\ell^1$.

Conversely, if we find a different representation $\alpha_n = \sum \tilde \beta_n^{(k)}$,
then subtracting it from the representation \eqref{1.4} and applying Lemma~\ref{L2.1} with $p=1$, we see that to each $\beta^{(l)} \notin \ell^1$ there corresponds a unique $\tilde\beta^{(k)}$ with the same phase, such that their difference is an $\ell^1$ sequence.
\end{remark}

The following lemma describes some properties of sequences of generalized bounded variation. In particular, it shows that real sequences of generalized bounded variation have, in essence, an even set of phases and a symmetric representation with respect to complex conjugation.

\begin{lemma}\label{L2.2} Let $\phi, \psi\in\mathbb{R}$,  $A,B,C\subset \mathbb{R}$, and $\beta = \{\beta_n\}_{n=N}^\infty$,  $\gamma = \{\gamma_n\}_{n=N}^\infty$ (with $N$ finite) complex sequences. Then
\begin{enumerate}[\rm (i)]
\item\label{L2.2(i)} If $\beta \in \GBV{\phi}$, then $\beta$ is bounded;

\item\label{L2.2(ii)} if $\beta \in \GBV{\phi}$, $\gamma \in \GBV{\psi}$, then $\{ \beta_n \gamma_n\}_{n=N}^\infty \in \GBV{ \phi+\psi}$

\item\label{L2.2(iii)} if $\beta \in \GBV{B}$, $\gamma \in \GBV{C}$, then $\{ \beta_n \gamma_n\}_{n=N}^\infty \in \GBV{ B+C}$

\item\label{L2.2(iv)} if $\beta \in \GBV{B}$, $\gamma \in \GBV{C}$, then $\{ \beta_n + \gamma_n\}_{n=N}^\infty \in \GBV{ B \cup C}$

\item\label{L2.2(v)} if $\beta \in\GBV{B}$, then $\bar \beta \in \GBV{-B}$

\item\label{L2.2(vi)} if $\{a_n-1\}_{n=1}^\infty \in \GBV{A}$, then $\{a_n^2-1\}_{n=1}^\infty \in \GBV{(A+A) \cup A}$

\item\label{L2.2(vii)} if $x \in \GBV{A}$ with $x_n\in\mathbb{R}$, then $x$ admits a representation 
\[
x = \sum_{l=1}^L (\beta^{(l)} + \bar \beta^{(l)})
\]
with $\beta^{(l)} \in \GBV{\phi_l}$, such that $\phi_l \in A$ and for every $\beta^{(l)}\notin\ell^1$, the corresponding $\phi_l$  is in $-A+2\pi\mathbb{Z}$.
\end{enumerate}
\end{lemma}

\begin{proof}
(i) follows from the triangle inequality,
\begin{align*}
\lvert \beta_n \rvert & \le \lvert e^{iN\phi} \beta_N \rvert  + \sum_{m=N}^{n-1} \lvert e^{i(m+1)\phi} \beta_{m+1} - e^{im\phi} \beta_m \rvert  \\
&  \le \lvert \beta_N \rvert  + \sum_{m=N}^\infty \lvert e^{i\phi} \beta_{m+1} - \beta_m \rvert
\end{align*}

(ii) follows from the triangle inequality and part (\ref{L2.2(i)}),
\begin{align*}
& \bigl\lvert e^{i(\phi+\psi)} \beta_{n+1} \gamma_{n+1} - \beta_{n} \gamma_{n} \bigr \rvert  \\
& \qquad\qquad \le  \bigl \lvert  e^{i\psi} \gamma_{n+1} (e^{i\phi} \beta_{n+1} - \beta_{n}) \bigr\rvert + \bigl\lvert \beta_{n}  (e^{i\psi} \gamma_{n+1} - \gamma_{n} ) \bigr \rvert   \\
& \qquad\qquad \le  \lVert \gamma \rVert_\infty \;  \bigl\lvert   e^{i\phi} \beta_{n+1} - \beta_{n} \bigr\rvert + \lVert \beta \rVert_\infty \;   \bigl\lvert e^{i\psi} \gamma_{n+1} - \gamma_{n} \bigr \rvert
\end{align*}
after summing over $n$.

(iii) is proved by decomposing $\beta$ and $\gamma$ into sequences of rotated bounded variation and applying (\ref{L2.2(ii)}).

(iv) and (v) follow directly from Definition~\ref{D1.1}.

(vi) follows from (\ref{L2.2(iii)}) and (\ref{L2.2(iv)}), using $a_n^2 - 1 = (a_n-1)^2 + 2(a_n-1)$.

(vii) Taking an arbitrary representation of $x$ and averaging it with its complex conjugate produces the desired form. Since $x=\bar x$, the other claim follows from (\ref{L2.2(v)}) and Remark~\ref{R2.1}.
\end{proof}

\section{Pr\" ufer Variables --- OPUC}

In this section we will define Pr\" ufer variables for OPUC and reduce the proof of Theorem~\ref{T1.1} to a criterion in terms of one of them. Pr\" ufer variables are named after Pr\"ufer \cite{Pru} who defined them for Sturm---Liouville operators.  The OPUC version of Pr\"ufer variables was first introduced by Nikishin \cite{Nik}, and later used by Nevai \cite{Nev2} and Simon \cite{OPUC2}.

For $z=e^{i\eta}$ with $\eta\in\mathbb{R}$, Pr\"ufer variables $r_n(\eta)$, $\theta_n(\eta)$ are defined by $r_n(\eta) > 0$, $\theta_n(\eta) \in\mathbb{R}$, and
\begin{equation}\label{3.1}
\varphi_n(e^{i\eta}) = r_n(\eta) e^{i[n\eta+\theta_n(\eta)]}
\end{equation}
(the ambiguity in $\theta_n$ modulo $2\pi$ is usually fixed by setting $\theta_0 = 0$ and $\lvert \theta_{n+1} - \theta_n \rvert < \pi$, but in this paper that will be irrelevant).

Then $\varphi_n^*(e^{i\eta}) = r_n(\eta) e^{-i \theta_n(\eta)}$ so
the Szeg\H o recursion relation \eqref{1.1} implies
\[
r_n e^{i[(n+1)\eta + \theta_n]} = \sqrt{1-\lvert \alpha_n \rvert^2} \, r_{n+1} e^{i [(n+1)\eta+ \theta_{n+1}]}  + \bar\alpha_n r_n e^{-i\theta_n}
\]
Regrouping and dividing by $\sqrt{1-\lvert \alpha_n \rvert^2} \, r_n e^{i[(n+1)\eta+\theta_n]}$ gives
\begin{equation}\label{3.2}
\frac{r_{n+1}}{r_n} e^{i  (\theta_{n+1}-\theta_n)} =  \frac{1 - \bar\alpha_n e^{-i[(n+1)\eta + 2 \theta_n]} }{ \sqrt{1-\lvert \alpha_n \rvert^2}}
\end{equation}

Part (\ref{L3.1(i)}) of the following lemma reduces the proof of Theorem~\ref{T1.1} to the proof of uniform convergence of $\log r_n(\eta)$ on intervals. Part (\ref{L3.1(ii)})  is also used in the proof of Theorem~\ref{T1.1}, to provide a contradiction in a crucial step.

\begin{lemma}\label{L3.1}
Let a measure $d\mu$ on the unit circle have Verblunsky parameters $\{\alpha_n\}_{n=0}^\infty$ and Pr\"ufer variables $r_n(\eta)$. Then
\begin{enumerate}[\rm (i)]
\item\label{L3.1(i)} If $B\subset \mathbb{R}$ is finite and  $\log r_n(\eta)$ converges uniformly on intervals $I$ with $\dist(I,B+2\pi\mathbb{Z})>0$, then Theorem~\ref{T1.1}{(\ref{T1.1(i)}),(\ref{T1.1(ii)})} hold, with the set $S$ given by $S = \{\exp(i \eta) \vert \eta \in  B\}$;
\item\label{L3.1(ii)} If $\alpha_n \to 0$, it is not possible for $\log r_n(\eta)$ to converge as $n\to \infty$ to $+\infty$ or $-\infty$ uniformly on an interval $I$.
\end{enumerate}
\end{lemma}

\begin{proof}  (i) Note that $r_n(\eta) = \lvert \varphi_n(e^{i\eta}) \rvert$, so using the Bernstein--Szeg\H o approximations (see \cite{Sexp}),
\begin{equation} \label{eq:approxUC}
\frac 1{2\pi} \frac{d\eta}{r_n^2(\eta)}  \buildrel{w}\over\to d\mu(e^{i\eta})
\end{equation}
Thus, if $\log r_n(\eta)$ converges uniformly on an interval $I$, then 
\[
d\mu(e^{i\eta}) = \frac 1{2\pi} \frac{1}{\lim\limits_{n\to \infty} r_n^2(\eta)} d\eta\quad\text{ on }I
\] 
This holds for any interval $I$ with $\dist(I,B+2\pi\mathbb{Z})>0$, and $\partial\mathbb{D} \setminus S$ can be covered by the images $J=\{e^{i\eta}\vert \eta\in I\}$ of countably many such intervals, which implies the conclusions of Theorem~\ref{T1.1}.

(ii)   If $r_n(\eta)$ converged uniformly to  $0$ or to $+\infty$ on $I$, \eqref{eq:approxUC} would imply that $\mu(I)=\infty$ or $\mu(I)=0$, contradicting either the assumption that $d\mu$ is a probability measure or a result of Geronimus \cite[Thm.\ 19.1]{Ger}
that $\alpha_n \to 0$ implies $\supp d\mu = \partial \mathbb{D}$.
\end{proof}

\section{Pr\" ufer Variables --- OPRL}

In this section we will define Pr\" ufer variables for OPRL and reduce the proof of Theorem~\ref{T1.2} to a criterion in terms of one of them. The OPRL analog of Pr\"ufer variables is known as the EFGP transform, by Eggarter, Figotin, Gredeskul, Pastur \cite{Egg, GP, PF} who developed and used it in the discrete Schr\"odinger case $a_n=1$. It was also extensively used by Kiselev--Last--Simon \cite{KLS}. For general OPRL, it was used by Breuer, Kaluzhny, Last, Simon \cite{KL,Bre,Br2,BLS}.

 For $x=2 \cos (\eta/2)$, $0<\eta<2\pi$, define $r_n(\eta) > 0$, $\theta_n(\eta) \in \mathbb{R}$ by
\begin{equation}\label{4.1}
r_n(\eta) e^{i[n \eta / 2 + \theta_n(\eta)]} =  a_n p_n(x) - p_{n-1}(x) e^{-i\eta/2} 
\end{equation}
Next we define 
\begin{equation} \label{4.2}
\alpha_n (\eta) = \frac{a_n^2 - 1 + e^{i \eta/2} b_{n+1}}{e^{i\eta} - 1}
\end{equation}
This variable will play the same role in our proof that Verblunsky coefficients $\alpha_n$ play for OPUC. In fact, after this section, we will not need to mention $a_n$ or $b_n$ individually, only their combination \eqref{4.2}. By decomposing $a_n^2-1$ and $b_n$ into sequences of rotated bounded variation, $\alpha_n(\eta)$ can be written as
\begin{equation}\label{4.3}
\alpha_n(\eta) = \sum_{l=1}^L h_l(\eta) \beta_n^{(l)}
\end{equation}
where $\beta^{(l)}$ has rotated bounded variation with phase $\phi_l$ and $h_l(\eta)$ are continuous non-vanishing functions on $(0,2\pi)$. In fact, $h_l(\eta)$ are either $1/(e^{i\eta}-1)$ or $e^{i\eta/2} / (e^{i\eta}-1)$, depending on whether the corresponding $\beta^{(l)}$ was a part of $\{a_n^2-1\}_{n=1}^\infty$ or $\{b_n\}_{n=1}^\infty$. Further, if $\{a_n^2-1\}_{n=1}^\infty, \{b_n\}_{n=1}^\infty \in \ell^p$, then $\beta^{(l)} \in\ell^p$ by Lemma~\ref{L2.1}.

Note that unlike in OPUC, an arbitrary choice of sequences $\beta^{(l)} \in \ell^p \cap \GBV{\phi_l}$ wouldn't correspond via \eqref{4.3} to a valid set of Jacobi parameters; rather, by Lemma~\ref{L2.2}(\ref{L2.2(vii)}), for each $\beta^{(l)}$, its complex conjugate is also one of the sequences in \eqref{4.3}.

Multiplying \eqref{4.1} by $e^{i\eta/2}$ gives
\begin{equation}\label{4.4}
r_n e^{i[(n+1) \eta / 2 + \theta_n]} =  a_n p_n e^{i\eta/2} - p_{n-1} 
\end{equation} 
Note that $2\Re \alpha_n = 1- a_n^2$ and $2\Re (\alpha_n e^{i\eta/2}) = b_{n+1}$ so using \eqref{4.4},
\begin{align*}
2 \Re \bigl( r_n e^{i[(n+1) \eta/2 + \theta_n]} \alpha_n  \bigr) & = 2 \Re \bigl( a_n p_n e^{i\eta/2} \alpha_n - p_{n-1} \alpha_n \bigr) \\
& = a_n p_n b_{n+1} + (a_n^2 - 1) p_{n-1}
\end{align*}
Subtracting this from \eqref{4.4}, then using the Jacobi recursion relation \eqref{1.2}, we have
\begin{align*}
r_n e^{i[(n+1) \eta / 2 + \theta_n]} - 2 \Re \bigl( r_n e^{i[(n+1) \eta/2 + \theta_n]} \alpha_n  \bigr) & = a_n ( a_{n+1} p_{n+1} - p_n e^{-i \eta/2} ) \\
& = a_n r_{n+1} e^{i[(n+1) \eta / 2 + \theta_{n+1}]} 
\end{align*}
where in the last line we used \eqref{4.1} with $n$ replaced by $n+1$. Dividing both sides by $a_n r_n e^{i[(n+1) \eta / 2 + \theta_n]}$ and again using  $a_n^2 = 1- 2 \Re \alpha_n$, we get
\begin{equation}\label{4.5}
\frac{r_{n+1}}{r_n} e^{i(\theta_{n+1}-\theta_n)} = \frac{1-\alpha_n - \bar \alpha_n e^{-i[(n+1)\eta+2\theta_n]}}{\sqrt{1-\alpha_n - \bar \alpha_n}}
\end{equation}

Part (\ref{L4.1(i)}) of the following lemma reduces the proof of Theorem~\ref{T1.2} to proving uniform convergence of $\log r_n(\eta)$ on intervals. Part (\ref{L4.1(ii)}) is also used in the proof of Theorem~\ref{T1.2}, to provide a contradiction in a crucial step.

\begin{lemma}\label{L4.1}
Let a measure $d\rho$ on the real line have Jacobi parameters $\{a_n, b_n\}_{n=1}^\infty$ with $a_n\to 1$, $b_n\to 0$ and Pr\"ufer variables $r_n(\eta)$. Then
\begin{enumerate}[\rm (i)]
\item\label{L4.1(i)} If $B\subset\mathbb{R}$ is finite, $0\in B$ and  $\log r_n(\eta)$ converges uniformly on intervals $I$ with  $\dist(I,B+2\pi\mathbb{Z})>0$, then Theorem \ref{T1.2}(\ref{T1.2(i)}),(\ref{T1.2(ii)}) hold, with the set $S$ given by $S = \{ 2\cos(\eta/2) \vert \eta \in B\}$;
\item\label{L4.1(ii)} It is not possible for $\log r_n(\eta)$ to converge as $n\to \infty$ to $+\infty$ or $-\infty$ uniformly on an interval $I$.
\end{enumerate}
\end{lemma}

\begin{proof} (i) We use a sequence of weak approximations to $d\rho$ (see \cite{Sexp})
\begin{equation}\label{4.6}
\frac {dx}{\pi (a_n^2 p_n^2(x) + p_{n-1}^2(x))}  \buildrel{w}\over\to d\rho(x)
\end{equation}
but we only know that with $x=2\cos(\eta/2)$,
\begin{equation}\label{4.7}
r_n^2(\eta) = a_n^2 p_n^2(x) - a_n x p_n(x) p_{n-1}(x) + p_{n-1}^2(x)
\end{equation}
uniformly converges on certain intervals.  For $\lvert x \rvert < 2- 2 \epsilon$ we have
\begin{equation} \label{4.8}
\epsilon (a_n^2 p_n^2(x) + p_{n-1}^2(x)) \le r_n^2(\eta) \le (2-\epsilon)  (a_n^2 p_n^2(x) + p_{n-1}^2(x))
\end{equation}
Let $I$ be an interval with $\dist(I,B+2\pi\mathbb{Z})>0$. Since $\log r_n$ converges uniformly on $I$, it is uniformly bounded on $I$. Since $0\in B$, $\dist(I,2\pi\mathbb{Z})>0$, so \eqref{4.8} implies $\log(a_n^2 p_n^2(x) + p_{n-1}^2(x))$ is uniformly bounded on $J=\{2\cos(\eta/2)\vert \eta\in I\}$. Thus, standard measure theory arguments applied to \eqref{4.6} imply that $d\rho(x) = f(x) dx$ on $J$ with $\log f$ bounded on $J$.


It remains to prove continuity of $f$ on $J$. By \cite[Thm.\ 4.2.13]{Nev}, since $a_n\to 1$ and $b_n\to 0$, for all bounded continuous real functions $h(x)$
\[
\lim_{n\to\infty} \int_{-\infty}^{+\infty} h(x) p_n(x) p_{n+k}(x) d\rho(x) = \frac 1{\pi} \int_{-2}^2 h(x) \frac{T_{\lvert k\rvert} (x/2)}{\sqrt{4-x^2}} \, dx
\]
where $T_k(x)$ are Chebyshev polynomials of the first kind, given by $T_k(\cos\theta) = \cos(k \theta)$. Using this and \eqref{4.7}, with $\eta(x)= 2 \arccos(x/2)$,
\begin{align*}
\lim_{n\to\infty} \int_{-\infty}^{+\infty} h(x) r_n(\eta(x))^2 d\rho(x) & = \frac 1{\pi} \int_{-2}^2 h(x) \frac{2 T_0(x/2) - x T_1(x/2)}{\sqrt{4-x^2}} \, dx \\
& = \frac 1{2\pi} \int_{-2}^2 h(x) \sqrt{4-x^2} \, dx
\end{align*}
Assuming in addition that $\supp h \subset J$, uniform convergence of $\log r_n(\eta)$ on $I$ implies
\begin{align*}
\lim_{n\to\infty} \int_{-\infty}^{+\infty} h(x) r_n^2(\eta(x))\, d\rho(x)  =  \int_{-\infty}^{+\infty} h(x) \lim_{n\to \infty} r_n^2(\eta(x)) \,  d\rho(x)
\end{align*}
Comparing the two gives
\begin{equation}
d\rho(x) = \frac 1{2\pi} \frac{\sqrt{4-x^2}}{ \lim\limits_{n\to\infty} r_n^2(2 \arccos(x/2))} \, dx \quad \text{on }J
\end{equation}
Since $(-2,2)\setminus S$ can be covered by countably many such intervals $J$, this concludes the proof.

(ii)  If $r_n(\eta)$ converged uniformly to $0$ or to $\infty$ on $I$, \eqref{4.8} and \eqref{4.6} would imply that $\rho(I)=\infty$ or $\rho(I)=0$. This would contradict either the assumption that $d\rho$ is a probability measure or a result of Blumenthal--Weyl \cite{Blu,Wey} (see also \cite[Sect.\ 1.4]{Rice})
that $a_n \to 1$, $b_n\to 0$ implies $\esssupp d\rho = [-2,2]$.
\end{proof}

\section{Equisummability}

In this section, we define a useful relation and present the framework for both OPRL and OPUC in a unified way. Define a constant $c$,
\begin{equation}\label{5.1}
c = \begin{cases}
0 & \text{for OPUC} \\
1 & \text{for OPRL}
\end{cases}
\end{equation}
Then \eqref{3.2} and \eqref{4.5} can be written in a unified way as
\begin{equation}\label{5.2}
\frac{r_{n+1}}{r_n} e^{i(\theta_{n+1}-\theta_n)} = \frac{1-c \alpha_n - \bar \alpha_n e^{-i[(n+1)\eta+2\theta_n]}}{\sqrt{(1-c\alpha_n)(1-c \bar \alpha_n) - \alpha_n \bar \alpha_n}}
\end{equation}
Taking the absolute value of this equation, or dividing it by its complex conjugate, we get
\begin{align}
\frac{r_{n+1}}{r_n} & = \frac{ \lvert 1 - \alpha_n  e^{i[(n+1)\eta + 2 \theta_n]}  - c \bar \alpha_n  \rvert }{ \sqrt{ (1 - c \alpha_n)(1   - c \bar \alpha_n) - \alpha_n \bar \alpha_n }  } \label{5.3}  \\
e^{2 i (\theta_{n+1} - \theta_n)} & = \frac{ 1 - \bar \alpha_n e^{-i[(n+1)\eta + 2 \theta_n]} - c \alpha_n }{1- \alpha_n e^{i[(n+1)\eta + 2 \theta_n]} - c \bar \alpha_n }  \label{5.4}
\end{align}
For both OPUC and OPRL, the sequence $\alpha(\eta)$ can be written as
\begin{equation}\label{5.5}
\alpha_n(\eta) = \sum_{l=1}^L h_l(\eta) \beta_n^{(l)} 
\end{equation}
where $\beta^{(l)}$ has rotated bounded variation with phase $\phi_l$, $\beta^{(l)} \in \ell^p$ and $h_l(\eta)$ are continuous non-vanishing functions away from $A_1+2\pi\mathbb{Z}$, with
\begin{equation}\label{5.6}
A_1 = \begin{cases}
\emptyset & \text{for OPUC} \\
\{0\} & \text{for OPRL}
\end{cases}
\end{equation}

For a given set $A$ of phases,  we will now define sets $A_p$ with $p$ a positive integer. Let
\begin{equation}\label{5.7}
A_2 = A \cup A_1 
\end{equation}
Let $q = \lceil (p-1)/2 \rceil$ (the smallest integer not smaller than $(p-1)/2$) and 
\begin{equation}  \label{5.8}
A_p = \begin{cases}  (\underbrace{A+\dots+A}_{q\text{ times}}) - (\underbrace{A+\dots+A}_{q-1\text{ times}}) & \text{for OPUC}  \\ 
 \underbrace{A_2+\dots+ A_2}_{p-1\text{ times}} &  \text{for OPRL}
  \end{cases} 
\end{equation}
For OPRL, note that Lemma~\ref{L2.2}(\ref{L2.2(vii)}) implies $A=-A$, and that $0\in A_2$, so the set $A_p$ contains all elements of
\[
(\underbrace{A+\dots+A}_{i\text{ times}}) - (\underbrace{A+\dots+A}_{j\text{ times}})
\]
for any $i\ge 1$, $j\ge 0$  and $i+j<p$. For OPUC, it only contains those with $i=j+1$.

\begin{defn}
Let $B\subset \mathbb{R}$ be a finite set. We define \emph{equisummability away from $B$}, a binary relation $\rel{B}$ on the set of sequences parametrized by $\eta\in\mathbb{R}$ by: $u_n(\eta) \rel{B}  v_n(\eta)$ if and only if
\begin{equation*}
\sum_{n=0}^{\infty} \bigl(u_n(\eta)-  v_n(\eta) \bigr)
\end{equation*}
converges uniformly (but not necessarily absolutely) in $\eta\in I$ for intervals $I$ with $\dist(I,B+2\pi \mathbb{Z})>0$.
\end{defn}

With this notation, if we are in the $\ell^p$ case, it suffices to show that
\begin{align}\label{5.9}
\log \frac{r_{n+1}(\eta)}{r_n(\eta)} & \rel{A_p}  0
\end{align}
because then Lemmas~\ref{L3.1}(\ref{L3.1(i)}) and \ref{L4.1}(\ref{L4.1(i)}) imply Theorems~\ref{T1.1} and \ref{T1.2}.

\section{Proof in the $\ell^2$ Case}

In this section, we present a proof of \eqref{5.9} in the $\ell^2$ case. We focus on this case in order to motivate elements of the proof of the general case, and in particular a key lemma. We remind the reader that for OPUC, the $\ell^2$ case has already been proved by Wong~\cite{Won}.

Taking the $\log$ of (\ref{5.3}) and expanding to linear order in $\alpha_n$, we get
\begin{equation*}
\log \frac{r_{n+1}}{r_n} = - \Re  \alpha_n e^{i[(n+1)\eta + 2 \theta_n]} + O(\lvert \alpha_n \rvert^2) 
\end{equation*}
In the $\ell^2$ case  $O(\lvert \alpha_n \rvert^2) \rel{A_1} 0$, so using \eqref{5.5},
\begin{align}
 \log \frac{r_{n+1}}{r_n} 
 & \rel{A_1}   -  \Re \sum_{l=1}^L  h_l(\eta) \beta_n^{(l)} e^{i[(n+1)\eta + 2\theta_n]} \label{6.1}
\end{align}

Now we need a way to control terms of the form $f(\eta) \Gamma_n e^{i[(n+1)\eta + 2 \theta_n]}$, with $\{ \Gamma_n \}$ of rotated bounded variation with phase $\phi$. But first, some definitions. We will need the function
\begin{equation}\label{6.2}
\chi(\eta) =   \frac 1{e^{- i \eta}-1} = - \frac 12 + \frac i2 \cot \frac{\eta}2
\end{equation}
Taylor expansions of \eqref{5.4} will turn out to be important: taking the $k$-th power of \eqref{5.4} and expanding in powers of $\alpha_n$, we have
\begin{equation}\label{6.3}
e^{2k i (\theta_{n+1}-\theta_n)} - 1 = P_{k,l}(\alpha_n,e^{i[(n+1)\eta+2\theta_n]}) + O(\lvert \alpha_n \rvert^l)
\end{equation}
where
\begin{align}
P_{k,l}(\alpha_n,e^{i[(n+1)\eta+2\theta_n]}) & = \!\!\!\! \sum_{\substack{u,v\ge 0 \\ 0<u+v<l}} \!\! \Bigl( (-1)^v \tbinom{k+u-1}{u} \tbinom{k}{v} (\alpha_n e^{i[(n+1)\eta+2\theta_n]} + c\bar \alpha_n)^u \nonumber \\
& \qquad \quad\qquad \times (\bar\alpha_n e^{-i[(n+1)\eta+2\theta_n]} + c\alpha_n)^v \Bigr) \label{6.4}
\end{align}

The first part of the following lemma will give us a way of passing from a sequence of the form $f(\eta) \Gamma_n e^{i[(n+1)\eta + 2 \theta_n]}$ to a faster decaying sequence, but at a cost of a multiplicative factor with possibly finitely many singularities. These singularities exactly correspond to the points where we can't rule out existence of a pure point. The main idea of the proof is that for $\eta$ away from $\phi$, the exponential factor $e^{i n\eta}$ in this sequence helps average out parts of it when partial sums are taken.

The second part of the lemma uses the $\ell^p$ condition and shows that it is allowed to replace an appearance of $e^{2ik(\theta_{n+1}-\theta_n)}-1$ by its Taylor polynomial $P_{k,l}$ of a sufficient power.

\begin{lemma}\label{L6.1}
Let $k\in\mathbb{Z}$ and $\phi\in [0,2\pi)$, with $k$ and $\phi$ not both equal to $0$. Let $B\subset \mathbb{R}$ be a finite set and $f: \mathbb{R} \setminus (B+2\pi\mathbb{Z}) \to \mathbb{C}$ be a continuous function such that $g(\eta)=f(\eta) \chi(k\eta-\phi)$ is also continuous on $\mathbb{R} \setminus (B+2\pi\mathbb{Z})$ (removable singularities in $g$ are allowed).

 If $\{\Gamma_n\}$ has rotated bounded variation with phase $\phi$ and $\Gamma_n \to 0$, then
\begin{equation}\label{6.5}
f(\eta) \Gamma_n  e^{i k [(n+1)\eta + 2 \theta_n]}   \rel{B}  g(\eta) \Gamma_n  e^{ik [(n+1)\eta + 2 \theta_n]}  \bigl( e^{2 i k ( \theta_{n+1} - \theta_n)}  - 1  \bigr)
\end{equation}
In particular, let $\Gamma_n = \beta_n^{(k_1)} \cdots \beta_n^{(k_s)} \bar\beta_n^{(l_1)} \cdots \bar \beta_n^{(l_t)}$ with  $\phi = \phi_{k_1}+\cdots + \phi_{k_s} - \phi_{l_1} - \cdots - \phi_{l_t}$. If all $\beta^{(j)}\in \ell^p$ and $A_1\subset B$, then
\begin{equation}\label{6.6}
f(\eta) \Gamma_n  e^{i k [(n+1)\eta + 2 \theta_n]}   \rel{B}  g(\eta) \Gamma_n  e^{ik [(n+1)\eta + 2 \theta_n]} P_{k,p-s-t}(\alpha_n,e^{i[(n+1)\eta+2\theta_n]})
\end{equation}
\end{lemma}

\begin{proof}
Start by substituting $f(\eta) = g(\eta)  (e^{-i(k \eta-\phi)} - 1)$,
\begin{align}
f(\eta) \Gamma_n  e^{i k [(n+1)\eta + 2 \theta_n]} & = g(\eta) (e^{-i(k \eta-\phi)} - 1) \Gamma_n  e^{i k [(n+1)\eta + 2 \theta_n]}  \nonumber \\
& = g(\eta) \bigl( e^{i\phi} \Gamma_n  e^{i k [n\eta + 2 \theta_n]}   - \Gamma_n  e^{i k [(n+1)\eta + 2 \theta_n]} \bigr) \label{6.7}
\end{align}
and note that $g(\eta)$ is bounded on intervals $I$ with $\dist(I,B+2\pi\mathbb{Z})>0$. 

For a sequence $x_n(\eta)$ which converges to $0$ uniformly in $\eta$ away from $B+2\pi\mathbb{Z}$,
\[
\sum_{n=0}^\infty \bigl(x_n(\eta) - x_{n+1}(\eta)\bigr)  =  x_0(\eta)
\]
uniformly in $\eta$, so $x_{n}(\eta) \rel{B} x_{n+1}(\eta)$. Taking $x_n(\eta) = e^{i\phi} \Gamma_{n}  e^{i k [n\eta + 2 \theta_{n}]}$ gives
\begin{equation} \label{6.8}
e^{i\phi} \Gamma_{n}  e^{i k [n\eta + 2 \theta_{n}]}  \rel{B} e^{i\phi} \Gamma_{n+1}  e^{i k [(n+1)\eta + 2 \theta_{n+1}]} 
\end{equation}
Meanwhile, the rotated bounded variation condition for $\Gamma_n$ implies
\begin{equation} \label{6.9}
e^{i\phi} \Gamma_{n+1}  e^{i k [(n+1)\eta + 2 \theta_{n+1}]}  \rel{B}  \Gamma_{n}  e^{i k [(n+1)\eta + 2 \theta_{n+1}]} 
\end{equation}
Applying \eqref{6.8} and then \eqref{6.9} to the first term of the right-hand side of \eqref{6.7} proves \eqref{6.5}.

To prove \eqref{6.6}, use Lemma~\ref{L2.2}(\ref{L2.2(ii)}),(\ref{L2.2(v)}) to note that $\Gamma$ has rotated bounded variation with phase $\phi$. Using \eqref{5.5} and continuity of $h_l(\eta)$ away from $A_1$, on an interval $I$ with $\dist(I,A_1+2\pi\mathbb{Z})>0$ we have
\begin{equation}\label{6.10}
\lvert \alpha_n \rvert \le C_1 \sum_{l=1}^L \lvert \beta_n^{(l)} \rvert
\end{equation}
for some constant $C_1$. Since $\beta^{(l)}$ are bounded sequences,  $\alpha_n(\eta)$ is uniformly bounded for $\eta\in I$. Thus, \eqref{6.3} implies
\begin{align*}
\bigl\lvert e^{2k i (\theta_{n+1}-\theta_n)} - 1 -  P_{k,p-s-t}(\alpha_n,e^{i[(n+1)\eta+2\theta_n]}) \bigr\rvert \le C_2 \lvert \alpha_n \rvert^{p-s-t}
\end{align*}
Combining this with \eqref{6.10} and $\Gamma_n = \beta_n^{(k_1)} \cdots \beta_n^{(k_s)} \bar\beta_n^{(l_1)} \cdots \bar \beta_n^{(l_t)}$, and using $\beta^{(j)}\in\ell^p$, we get
\[
g(\eta) \Gamma_n  e^{ik [(n+1)\eta + 2 \theta_n]} \bigl( e^{2k i (\theta_{n+1}-\theta_n)} - 1 -  P_{k,p-s-t}(\alpha_n,e^{i[(n+1)\eta+2\theta_n]}) \bigr)  \rel{B} 0
\]
Subtracting this from \eqref{6.5} gives \eqref{6.6} and completes the proof.
 \end{proof}

Using this lemma, we can finish the proof for the $\ell^2$ case. Notice that the factor $\chi(\eta-\phi_l)$ is continuous away from $\phi_l \in A_2$, and that  $h_l(\eta)$ are continuous away from $A_1\subset A_2$. Also, from \eqref{6.4} or \eqref{5.4}, \eqref{6.3} we have
$e^{2 i ( \theta_{n+1} - \theta_n)}  - 1  =  O(\lvert \alpha_n\rvert)$, i.e. $P_{1,1}=0$, so by Lemma~\ref{L6.1},
\begin{equation}\label{6.11}
h_l(\eta) \beta_n^{(l)} e^{i[(n+1)\eta + 2\theta_n]} \rel{A_2}  0
\end{equation}
Summing this over $l$ and combining into \eqref{6.1} finally gives
\[
\log \frac{r_{n+1}}{r_n} \rel{A_2}  0
\]
which completes the proof.

\section{Proof in the $\ell^3$ Case}

In this section, we present the proof in the $\ell^3$ case to provide further motivation for the general proof. Beyond $\ell^2$, Lemma~\ref{L6.1} needs to be used iteratively, and the $\ell^3$ case illustrates the difficulties encountered in performing this iterative procedure.

Taking the $\log$ of \eqref{5.3} and expanding in powers of $\alpha_n$, then using $O(\lvert \alpha_n\rvert^3) \rel{A_1} 0$ implies
\begin{align} 
 \log \frac{r_{n+1}}{r_n}  &  \rel{A_1} \Re \bigl( - \alpha_n e^{i[(n+1)\eta + 2 \theta_n]}  - \tfrac 12 \alpha_n^2 e^{2i[(n+1)\eta + 2 \theta_n]} \nonumber \\
 & \qquad\qquad\qquad\quad  -  c \alpha_n \bar  \alpha_n e^{i[(n+1)\eta + 2 \theta_n]}   +  \tfrac 12 \alpha_n  \bar\alpha_n \bigr)  \label{7.1}
\end{align}
As in the $\ell^2$ case, we now want to apply Lemma~\ref{L6.1} to parts of this expression. We begin with the first-order term in $\alpha_n$. In the $\ell^2$ case, using \eqref{5.5} to break up $\alpha_n$ and using Lemma~\ref{L6.1} gave \eqref{6.11}. However, applying the same lemma in the $\ell^3$ case, we need $P_{1,2}$ instead of $P_{1,1}$, since terms quadratic in the sequences $\beta^{(j)}$ cannot be automatically discarded. Thus, instead of \eqref{6.11} we get
\begin{align}
h_l(\eta) \beta_n^{(l)} e^{i[(n+1)\eta + 2\theta_n]} & \rel{A_2}  h_l(\eta) \chi(\eta-\phi_l) \beta_n^{(l)} e^{i[(n+1)\eta + 2\theta_n]}  \bigl( - c\alpha_n +c \bar \alpha_n  \nonumber \\
& \qquad\; - \bar \alpha_n e^{-i[(n+1)\eta+2\theta_n]}  + \alpha_n e^{i[(n+1)\eta+2\theta_n]}   \bigr) \label{7.2}
\end{align}
Note that all terms on the right-hand side contain a  $\beta_n^{(l)}$ and an $\alpha_n$ or $\bar \alpha_n$, so we have obtained a faster decaying expression in $n$, although at the cost of a singularity at $\eta = \phi_l$.

Summing \eqref{7.2} over $l$ and inserting into \eqref{7.1}, and using \eqref{5.5} to replace $\alpha_n$ everywhere, we have
\begin{align} 
& \log \frac{r_{n+1}}{r_n}  \rel{A_2} \Re \sum_{l,m=1}^L \bigl( X_{l,m} + Y_{l,m} + Z_{l,m}  + T_{l,m} \bigr) \label{7.3}
 \end{align}
 where
 \begin{align} 
X_{l,m} & =  - \bigl( \tfrac 12 + \chi(\eta - \phi_l) \bigr) h_l(\eta) h_m(\eta) \beta_n^{(l)} \beta_n^{(m)} e^{2i[(n+1)\eta + 2\theta_n]}  \label{7.4} \\
Y_{l,m} & =   \bigl( \tfrac 12 + \chi(\eta - \phi_l) \bigr) h_l(\eta) \bar h_m(\eta) \beta_n^{(l)} \bar \beta_n^{(m)} \label{7.5} \\
Z_{l,m} & =  c  \chi(\eta - \phi_l)  h_l(\eta) h_m(\eta) \beta_n^{(l)} \beta_n^{(m)} e^{i[(n+1)\eta + 2\theta_n]} \label{7.6} \\
T_{l,m} & =  - c \bigl(  1 + \chi(\eta - \phi_l) \bigr) h_l(\eta) \bar h_m(\eta) \beta_n^{(l)} \bar \beta_n^{(m)} e^{i[(n+1)\eta + 2\theta_n]} \label{7.7}
 \end{align}
 We proceed by applying Lemma~\ref{L6.1} to these expressions. 
 
 For OPRL, since singularities of $\chi(\eta-\phi_l-\phi_m)$ and $\chi(\eta-\phi_l + \phi_m)$ are inside $A_3$,  applying Lemma~\ref{L6.1} we get
\begin{align}
Z_{l,m} \rel{A_3} 0 \label{7.8} \\
T_{l,m} \rel{A_3} 0  \label{7.9}
\end{align}
The same formulas hold for OPUC, but for a different reason: $c=0$ implies that $Z_{l,m}=T_{l,m}=0$, so \eqref{7.8} and \eqref{7.9} are trivial. This is why for OPUC, $\phi_l+\phi_m$ and $\phi_l-\phi_m$ don't need to be included into $A_3$.

For $X_{l,m}$, Lemma~\ref{L6.1} gives a multiplicative factor $\chi(2\eta - \phi_l - \phi_m)$, which has singularities at $\eta = (\phi_l+\phi_m)/2 + \pi \mathbb{Z}$. These points are not in $A_3$, so it might seem that we will have to apply Lemma~\ref{L6.1} with a set greater than $A_3$. We are saved by the observation
\begin{align}
\bigl( 1+\chi(\eta-\phi_l) + \chi(\eta-\phi_m) \bigr) \chi(2\eta-\phi_l-\phi_m) =  \chi(\eta-\phi_l) \chi(\eta-\phi_m)  \label{7.10}
\end{align}
which is straightforward to check from \eqref{6.2}. Thus, applying Lemma~\ref{L6.1}  to $X_{l,m}+X_{m,l}$, the points $\eta = (\phi_l+\phi_m)/2 + \pi \mathbb{Z}$ are just removable singularities in \eqref{7.10} and we get
\begin{align}
X_{l,m}+X_{m,l} & \rel{A_2} 0 \label{7.11}
\end{align}
Since \eqref{7.3} contains a sum over all $l,m$, this is sufficient for our purposes. Combining terms with different permutations of the same indices will also be used in the general case, to avoid unnecessarily expanding the set of critical points. Indeed, Section~8 generalizes the observation \eqref{7.10} to the general case.

 When $\phi_l \neq \phi_m$, $\chi(\phi_m-\phi_l)$ is just a finite constant so Lemma~\ref{L6.1} can be applied to $Y_{l,m}$ to give
\begin{align}
Y_{l,m} & \rel{A_2} 0 \qquad \text{(when $\phi_l\neq \phi_m$)} \label{7.12}
\end{align}

Combining \eqref{7.8}, \eqref{7.9}, \eqref{7.11} and \eqref{7.12} into \eqref{7.3}, we have
\begin{align}
\log\frac{r_{n+1}}{r_n} &  \rel{A_3} \Re \sum_{\substack{ 1 \le l,m \le L \\ \phi_l=\phi_m}} Y_{l,m}   \label{7.13}
\end{align}
Lemma~\ref{L6.1} is not applicable to the remaining $Y_{l,m}$'s, but we are again saved by an observation that
\begin{equation}\label{7.14}
\Re \bigl( \tfrac 12 + \chi(\eta-\phi_l) \bigr) = 0 
\end{equation}
Because of this, when $\phi_l=\phi_m$,
\begin{align*}
\bar Y_{l,m} =  - \bigl( \tfrac 12 + \chi(\eta - \phi_l) \bigr) \bar h_l(\eta)  h_m(\eta) \bar \beta_n^{(l)} \beta_n^{(m)}  =   - Y_{m,l}  
\end{align*}
so $\Re (Y_{l,m} + Y_{m,l}) = 0$ and \eqref{7.13} becomes
\begin{equation} \label{7.15}
\log\frac{r_{n+1}}{r_n} \rel{A_3}  0
\end{equation}
which completes the proof.

In the proof above the observation \eqref{7.14} was crucial. To try to arrive to a more illuminating proof, lets focus on OPUC (where $h_l(\eta)=1$) and assume that instead of \eqref{7.13} we have, more generally,
\begin{equation}\label{7.16}
\log\frac{r_{n+1}}{r_n}  \rel{A_3} \Re \sum_{\substack{ 1 \le l,m \le L \\ \phi_l=\phi_m}} f_l(\eta) \beta_n^{(l)} \bar \beta_n^{(m)} 
\end{equation}
We will now show that $\Re f_l(\eta)=0$ for all $l$ and $\eta$ by proving that the converse leads to a contradiction with Lemma~\ref{L3.1}(\ref{L3.1(ii)}).

Assume  $\Re f_k(\eta_0)\neq 0$ for some $k$ and $\eta_0$. Let
\begin{equation}\label{7.17}
\beta_n^{(l)} = \begin{cases} e^{- i n \phi_k}/(n+2)^{1/2} & \text{for } l=k \\ 0 & \text{else} \end{cases}
\end{equation}
We have suppressed all $\beta^{(l)}$ with $l\neq k$. We have chosen $n+2$ in order to make all $\lvert\beta_n^{(k)} \rvert <1$; note that this makes $\alpha_n = \beta_n^{(k)}$ an allowed choice of Verblunsky coefficients, corresponding by Verblunsky's theorem to a unique probability measure on the unit circle.

With the choice \eqref{7.17}, \eqref{7.16} becomes
\begin{equation}\label{7.18}
\log\frac{r_{n+1}}{r_n}  \rel{A_3} \Re f_k(\eta) / (n+2)
\end{equation}
Since the harmonic series is divergent and $\Re f_k(\eta)$ is continuous in $\eta$, depending on the sign of $\Re f_k(\eta_0)$, summing \eqref{7.18} in $n$ gives
\[
\log r_n(\eta) \to \pm \infty
\]
uniformly in a neighborhood of $\eta_0$. However, this is a contradiction with Lemma~\ref{L3.1}(\ref{L3.1(ii)}). Thus, $\Re f_l(\eta) = 0$, so \eqref{7.16} becomes \eqref{7.15}, which completes this alternative proof for OPUC. This method can be applied to OPRL as well, with one extra difficulty: $\beta^{(l)}$'s are not independent there, so constructing counterexamples we have to be more careful than \eqref{7.17}. Indeed, instead of relying on observations of the type \eqref{7.14}, this will be the method we  will apply to the general $\ell^p$ case in Section 9.

\section{Narrowing the Set of Possible Pure Points}

In the previous section, if we hadn't made the observation \eqref{7.10} telling us that $\eta = \frac{\phi_k+\phi_l}2 + \pi \mathbb{Z}$ are removable singularities, we would have only proved equisummability away from a larger set of points, and we would have had a weaker result on the set of possible pure points. In this section, we generalize that observation to $\ell^p$. In the $\ell^p$ case, iterations of Lemma~\ref{L6.1} give multiplicative factors of the form
\[
\chi \biggl( k \eta - \sum_{a=1}^i \phi_{m_a}+ \sum_{b=1}^j \phi_{n_b} \biggr)
\]
with $k\le i$ and $i+j<p$. Such a factor has singularities at
\begin{equation}\label{8.1}
\eta = \frac 1k \biggl( \sum_{a=1}^i \phi_{m_a} - \sum_{b=1}^j \phi_{n_b} \biggr) + \frac 1k  \, 2\pi\mathbb{Z}
\end{equation}
Surprisingly, with a more careful analysis shown in this section, all the singularities corresponding to $k\ge 2$ will turn into removable singularities where needed, so they don't have to be included into $A_p$. 

The analysis that follows is quite technical, but the reader not interested in this aspect of the results may skip to the next section and replace the set $A_p$ by a greater (but still finite) set, containing all elements of the form \eqref{8.1} with $k\le i$ and $i+j<p$.

First let us set some conventions and definitions. We will use the Kronecker symbol $\delta_n$ which is $1$ if $n=0$ and $0$ otherwise. Note that
\begin{equation}\label{8.2}
\sum_{i=0}^I \delta_{i-k}  \delta_{I-i-(K-k)} = \delta_{I-K} 
\end{equation}
We will use the combinatorial convention for binomial coefficients, i.e.
\begin{equation}\label{8.3}
\binom{n}{k} = \begin{cases} \frac{n!}{k! (n-k)!} & \text{if }0\le k \le n \\
0 & \text{else} \end{cases}
\end{equation}
Two identities will be useful: for $l,m,n\ge 0$,
\begin{align}
\sum_{k=0}^l \binom{m}{k} \binom{n}{l-k} & = \binom{m+n}{l} \label{8.4} \\
\sum_{k=0}^l \binom{m+k}{m} \binom{n+l-k}{n} & = \binom{l+m+n+1}{m+n+1} \label{8.5}
\end{align}
\eqref{8.4} is just Vandermonde's identity. The more obscure \eqref{8.5} has a combinatorial proof, by double-counting the number of subsets of $\{1,\dots,l+n+m+1\}$ with exactly $m+n+1$ elements: observe that the number of such subsets whose $(m+1)$-st smallest element is $m+k+1$ is exactly $\binom{m+k}{m} \binom{n+l-k}{n}$.

We also need a kind of symmetrized product of functions:

\begin{defn}\label{D8.1}
For a function $p_{I,J}$ of $1+I+J$ variables and a function $q_{K,L}$ of $1+K+L$ variables, we define their symmetric product as a function $p_{I,J}\odot q_{K,L}$ of $1+(I+K)+(J+L)$ variables by
\begin{align*}
 (p_{I,J}\odot q_{K,L}) \bigl(\eta; \{x_i\}_{i=1}^{I+K} ; \{y_j\}_{j=1}^{J+L}\bigr) & = \frac{1}{(I+K)! (J+L)!} \sum_{\substack{\sigma\in S_{I+K} \\ \tau \in S_{J+L}}} r_{\sigma,\tau}
 \end{align*}
 with $S_n$ the symmetric group in $n$ elements and
 \begin{align*}
 r_{\sigma,\tau} =  p_{I,J} \bigl(\eta; \{x_{\sigma(i)}\}_{i=1}^{I} ; \{y_{\tau(j)}\}_{j=1}^{J}\bigr)  q_{K,L} \bigl(\eta; \{x_{\sigma(i)}\}_{i=I+1}^{I+K} ; \{y_{\tau(j)}\}_{j=J+1}^{J+L}\bigr)
\end{align*}
\end{defn}

It is straightforward to see that $\odot$ is commutative and associative.

Assuming we are in the $\ell^p$ case, expanding the $\log$ of \eqref{5.3} in powers of $\alpha_n$ and using $O(\lvert \alpha_n\rvert^p) \rel{A_1} 0$ gives
\begin{align}
\log \frac{r_{n+1}}{r_n} 
	& \rel{A_1}  - \Re \sum_{\substack{K,L\ge 0 \\ 0 <  K+L < p}} \tfrac 1{K+L} \tbinom{K+L}{K}  \bigl( \alpha_n e^{i[(n+1)\eta+2\theta_n]} \bigr)^K   (c \bar\alpha_n)^L  \nonumber   \\
	& \qquad\;\;  + \tfrac 12  \sum_{\substack{k,l \ge 0 \\ 0 < k+2l <p}} \tfrac 1{k+l} \tbinom{k+l}{k} (c\alpha_n+ c\bar \alpha_n)^k  \bigl((1-c^2) \alpha_n \bar \alpha_n \bigr)^l   \label{8.6}
\end{align}
Note that this is of the form
\begin{align} \label{8.7}
\log  \frac{r_{n+1}}{r_n}  & \rel{A_1}  \Re \sum_{\substack{I,J,K,L\ge 0 \\ I+J< p}}  \xi_{I,J,K,L} \,  \alpha_n^I \bar\alpha_n^J e^{iK[(n+1)\eta+2\theta_n]} c^L
\end{align}
where $\xi_{I,J,K,L}$ are constants. For $K>0$ only the first sum in \eqref{8.6} contributes to  $\xi_{I,J,K,L}$  and we read off their values,
\begin{equation}\label{8.8}
\xi_{I,J,K,L}  =  \delta_{I-K} \delta_{J-L} \frac 1{K+L} \binom{K+L}{K} \qquad\text{(for $K>0$)}
\end{equation}
(the values for $K=0$ will turn out to be of no importance to us). 

Our method is to substitute $\alpha_n$ using \eqref{5.5} and apply Lemma~\ref{L6.1} to terms of the form
\begin{equation} \label{8.9}
f(\eta) \; \prod_{i=1}^I (h_{k_i}(\eta) \beta_n^{(k_i)}) \; \prod_{j=1}^J (\bar h_{l_j}(\eta) \bar \beta_n^{(l_j)}) \; e^{iK[(n+1)\eta+2\theta_n]} \; c^L
\end{equation}
in increasing order of $I+J$. Note that this term will occur in all possible permutations of $k_1,\dotsc, k_I$ and of $l_1,\dotsc, l_J$, so we can average in those terms before applying Lemma~\ref{L6.1}. After such averaging, the function $f(\eta)$ in the term \eqref{8.9} is of the form
\[
f_{I,J,K,L}(\eta;\phi_{k_1},\dotsc,\phi_{k_I};\phi_{l_1},\dotsc,\phi_{l_J})
\]
and the corresponding $g(\eta)$ constructed by Lemma~\ref{L6.1} is
\begin{equation}\label{8.10}
g_{I,J,K,L} =\chi\biggl(  K \eta - \sum_{i=1}^I \phi_{k_i} + \sum_{j=1}^J \phi_{l_j} \biggr)  f_{I,J,K,L}
\end{equation}
All terms we encounter have $I,J,K,L\ge 0$, so we define 
\begin{equation}\label{8.11}
f_{I,J,K,L} = g_{I,J,K,L} = 0 \qquad\text{unless }I,J,K,L\ge 0
\end{equation}
Note that $f_{I,J,K,L}$ and $g_{I,J,K,L}$ are well-defined functions of $1+I+J$ parameters, and that they are symmetric in the $I$ parameters $\phi_{k_i}$ and also in the $J$ parameters $\phi_{l_j}$. 
Our goal is precisely to show that  $g_{I,J,K,L}$ has its singularities only at points of the form \eqref{8.1} with $k=1$. To do this, we will first establish a recurrence relation for these functions.

Any contribution to $f_{I,J,K,L}$ is either $\xi_{I,J,K,L}$ from the starting expression \eqref{8.7} or comes from an earlier term as $g_{\iota,j,k,l}$ multiplied by a constant from the Taylor expansion $P_{k,p-\iota-j}$ of $e^{2ik (\theta_{n+1}-\theta_n)} - 1$. Starting from \eqref{6.4} and expanding, we have
\begin{align}
\!\!\!\! P_{k,l}(\alpha_n,e^{i[(n+1)\eta+2\theta_n]}) & = \!\!\!\!\!\!\! \sum_{\substack{\alpha,\beta,\gamma,\delta\ge 0 \\ 0<\alpha+\beta+\gamma+\delta<l}} \!\!\!\! \Bigl( (-1)^{\gamma+\delta} \tbinom{k+\alpha+\beta-1}{\alpha+\beta} \tbinom{\alpha+\beta}{\alpha} \tbinom{k}{\gamma+\delta} \tbinom{\gamma+\delta}{\gamma}  \nonumber \\
& \;\, \times  (\alpha_n)^{\alpha+\delta} (\bar \alpha_n)^{\beta+\gamma} (e^{i[(n+1)\eta+2\theta_n]})^{\alpha-\gamma} c^{\beta+\delta}\Bigr) \label{8.12}
\end{align}
From \eqref{8.12} we read off the value of the constant multiplying $g_{\iota,j,k,l}$, and matching the powers of $\alpha_n$, $\bar \alpha_n$, $e^{i[(n+1)\eta+2\theta_n]}$, and $c$, we get $I=\iota+\alpha+\delta$, $J=j+\beta+\gamma$, $K=k+\alpha-\gamma$, $L=l+\beta+\delta$.

Since $f_{I,J,K,L}$ is then symmetrized in the appropriate variables, every product of $g_{\iota,j,k,l}$ by a constant becomes a symmetric product, so 
\begin{equation}\label{8.13}
\!\!\! f_{I,J,K,L} = \xi_{I,J,K,L}   + \!\!\!\! \sum_{\substack{\alpha,\beta,\gamma,\delta \ge 0 \\ \alpha+\beta+\gamma+\delta\ge 1}} \!\!\!\! \omega_{K,\alpha,\beta,\gamma,\delta} \odot g_{I-\alpha-\delta,J-\beta-\gamma,K+\gamma-\alpha,L-\beta-\delta} 
\end{equation}
with $\omega_{K,\alpha,\beta,\gamma,\delta}$ a constant function of $1+(\alpha+\delta)+(\beta+\gamma)$ variables,
\begin{equation}\label{8.14}
\omega_{K,\alpha,\beta,\gamma,\delta} = (-1)^{\gamma+\delta} \tbinom{K+\gamma+\beta-1}{\alpha+\beta} \tbinom{K+\gamma-\alpha}{\gamma+\delta}  \tbinom{\alpha+\beta}{\alpha} \tbinom{\gamma+\delta}{\gamma}
\end{equation}
(this is the constant from \eqref{8.12}, with the replacement $k=K+\gamma-\alpha$). By the convention \eqref{8.3}, the right-hand side of \eqref{8.14} is $0$ unless $K\ge 1$ and $\alpha,\beta,\gamma,\delta\ge 0$.

We have found the desired recursion relation, in the form of \eqref{8.13}. Note that \eqref{8.10}, \eqref{8.11} and \eqref{8.13} determine the $f_{I,J,K,L}$ and $g_{I,J,K,L}$ uniquely.

Since $\omega_{K,0,0,0,0} = 1$, it is convenient to define
\begin{equation}\label{8.15}
h_{I,J,K,L} = f_{I,J,K,L}+g_{I,J,K,L}
\end{equation}
 and rewrite \eqref{8.13} as
\begin{equation}\label{8.16}
\!\!\!\!\!\!\! h_{I,J,K,L} = \xi_{I,J,K,L}   +\!\!\!\! \sum_{\alpha,\beta,\gamma,\delta \ge 0}\!\!\! \omega_{K,\alpha,\beta,\gamma,\delta} \odot g_{I-\alpha-\delta,J-\beta-\gamma,K+\gamma-\alpha,L-\beta-\delta} 
\end{equation}
Note that \eqref{8.15} and \eqref{8.10} imply
\begin{equation}\label{8.17}
h_{I,J,K,L}  = g_{I,J,K,L} \, \exp \Bigl(- i \bigl( K \eta - \sum_{i=1}^I  \phi_{k_i} + \sum_{j=1}^J  \phi_{l_j} \bigr) \Bigr) 
\end{equation}

It will be useful to introduce a rescaled version of functions introduced so far.

Define $\Omega_{K,\alpha,\beta,\gamma,\delta}$ as a function of $1+(\alpha+\delta)+(\beta+\gamma)$ variables,
\begin{equation}\label{8.18}
\Omega_{K,\alpha,\beta,\gamma,\delta}  = (-1)^{\gamma+\delta} \tbinom{K+\gamma+\beta-1}{K-1} \tbinom{K}{\alpha+\delta}  \tbinom{\alpha+\delta}{\alpha} \tbinom{\beta+\gamma}{\beta} 
\end{equation}
By \eqref{8.3}, this is equal to $0$ unless $K\ge 1$ and $\alpha,\beta,\gamma,\delta\ge 0$.

Define  $\Xi_{I,J,K,L}$ as a function of $1+I+J$ variables equal to
\begin{equation} \label{8.19}
\Xi_{I,J,K,L}   =  \delta_{I-K} \delta_{J-L} \tbinom{K+L-1}{K-1}
\end{equation}
By \eqref{8.3}, this is equal to $0$ unless $I=K\ge 1$ and $J=L\ge 0$.

It is straightforward to check
\begin{align}
(K+\gamma-\alpha) \Omega_{K,\alpha,\beta,\gamma,\delta} & = K \omega_{K,\alpha,\beta,\gamma,\delta} 
 \\
\Xi_{I,J,K,L} & = K  \xi_{I,J,K,L} 
\end{align}
so if we define
\begin{align}
G_{I,J,K,L} & = K g_{I,J,K,L} \label{8.22} \\
H_{I,J,K,L} & = K h_{I,J,K,L} \label{8.23}
\end{align}
then multiplying \eqref{8.16} and \eqref{8.17} by $K$ gives
\begin{align}
& \!\!\!\! H_{I,J,K,L} = \Xi_{I,J,K,L}   + \!\!\!\! \sum_{\alpha,\beta,\gamma,\delta \ge 0} \!\!\!\! \Omega_{K,\alpha,\beta,\gamma,\delta} \odot G_{I-\alpha-\delta,J-\beta-\gamma,K+\gamma-\alpha,L-\beta-\delta} \label{8.24} \\
& \!\!\!\! H_{I,J,K,L} = G_{I,J,K,L} \, \exp \Bigl(- i \bigl( K \eta - \sum_{i=1}^I  \phi_{k_i} + \sum_{j=1}^J  \phi_{l_j} \bigr) \Bigr) \label{8.25} 
\end{align}
We are striving to prove the identity
\begin{equation}\label{8.26}
\sum_{i,j,l\ge 0} G_{i,j,k,l} \odot G_{I-i,J-j,K-k,L-l} =  \begin{cases} G_{I,J,K,L} & \text{if } 0<k<K \\
0 & \text{else} \end{cases}
\end{equation}
Comparing with the $\ell^3$ case, the observation \eqref{7.10} is a special case of this identity, namely, $G_{2,0,2,0} = G_{1,0,1,0} \odot G_{1,0,1,0}$ (since $G_{0,0,1,0}=0$ is easily computed from the recurrence relations).

The following lemma proves identity \eqref{8.26} and uses it to describe non-removable singularities of $f_{I,J,K,L}$ and $g_{I,J,K,L}$. It also analyzes the case $L=0$ in particular, since this is the only case that matters for OPUC ($c=0$ means that \eqref{8.9} vanishes for $L>0$).

\begin{lemma}\label{L8.1}
For $I,J,K,L,k,A,B,C,D\in\mathbb{Z}$, the following are true: 
\begin{enumerate}[\rm (i)]
\item For $0<k<K$,
\begin{equation}\label{8.27}
\sum_{i=0}^I \sum_{j=0}^J \sum_{l=0}^L \Xi_{i,j,k,l} \odot \Xi_{I-i,J-j,K-k,L-l}  = \Xi_{I,J,K,L}
\end{equation}

\item For $0<k<K$,
\begin{equation}\label{8.28}
\sum_{a=0}^A \sum_{b=0}^B \sum_{c=0}^C \sum_{d=0}^D \Omega_{K-k,A-a,B-b,C-c,D-d} \odot \Omega_{k,a,b,c,d} = \Omega_{K,A,B,C,D}
\end{equation}

\item For $k\ge 1$, 
\begin{align}
 &\!\! \sum_{i,j,l\ge 0} \Xi_{i,j,k,l} \odot G_{I-i,J-j,K-k,L-l} \nonumber \\
 & \quad\qquad\qquad = \sum_{\substack{\alpha,\beta,\gamma,\delta\ge 0 \\ \alpha \ge \gamma+k}}  \Omega_{k,\alpha,\beta,\gamma,\delta} \odot G_{I-\alpha-\delta,J-\beta-\gamma,K+\gamma-\alpha,L-\beta-\delta} \label{8.29}
\end{align}

\item \eqref{8.26} holds for all $I,J,K,L\in \mathbb{Z}$.

\item\label{L8.1(v)} Non-removable singularities of $f_{I,J,K,L}$ are of the form \eqref{8.1} with $k=1$ and $i+j< I+J$.

\item\label{L8.1(vi)} Non-removable singularities of $g_{I,J,K,L}$ are of the form \eqref{8.1} with $k=1$ and $i+j\le I+J$.

\item\label{L8.1(vii)} Non-removable singularities of $f_{I,J,K,0}$ are of the form \eqref{8.1} with $k=i-j=1$ and $i+j< I+J$.

\item\label{L8.1(viii)} Non-removable singularities of $g_{I,J,K,0}$ are of the form \eqref{8.1} with $k=i-j=1$ and $i+j\le I+J$.
\end{enumerate}
\end{lemma}

\begin{proof}
(i) First note that both sides of \eqref{8.27} are zero unless $I,J,L\ge 0$. If $I,J,L\ge 0$, using the definition \eqref{8.19}, \eqref{8.27} follows from a double application of \eqref{8.2} to resolve the sums in $i$ and $j$, and \eqref{8.5} to resolve the sum in $l$.

(ii) First note that both sides of \eqref{8.28} are zero unless $A,B,C,D\ge 0$. If $A,B,C,D\ge 0$, using the definition \eqref{8.18}, the left-hand side of \eqref{8.28} becomes a product of a sum in indices $a$ and $d$ and a sum in $b$ and $c$.

For the sum in $a$ and $d$, we introduce a change of indices to $x=a+d$ instead of $d$. Since the summand is $0$ outside the limits of summation, including some extra terms doesn't alter the sum, so
\begin{align*}
& \sum_{a=0}^A \sum_{d=0}^D \tbinom{K-k}{A+D-a-d} \tbinom{A+D-a-d}{A-a}\tbinom{k}{a+d} \tbinom{a+d}{a}  \\
& \qquad\qquad\qquad = \sum_{x=0}^{A+D} \sum_{a=0}^x \tbinom{K-k}{A+D-x} \tbinom{A+D-x}{A-a}\tbinom{k}{x} \tbinom{x}{a} \\
& \qquad\qquad\qquad = \tbinom{K}{A+D} \tbinom{A+D}{A}
\end{align*}
after a double application of \eqref{8.4}, first to compute the sum in $a$, and then to compute the sum in $x$.

 In the sum over $b$ and $c$, we introduce a change of indices to $y=b+c$ instead of $c$. Analogously to the previous sum, since the summand is $0$ outside the limits of summation, 
 \begin{align*}
& \sum_{b=0}^B \sum_{c=0}^C \tbinom{K-k+B-b+C-c-1}{K-k-1} \tbinom{B+C-b-c}{B-b} \tbinom{k+c+b-1}{k-1} \tbinom{b+c}{b} \\
& \qquad\qquad\qquad = \sum_{y=0}^{B+C} \sum_{b=0}^y \tbinom{K-k+B+C-y-1}{K-k-1} \tbinom{B+C-y}{B-b} \tbinom{k+y-1}{k-1} \tbinom{y}{b} \\
& \qquad\qquad\qquad = \tbinom{K+B+C-1}{K-1} \tbinom{B+C}{B}
\end{align*}
where we have used \eqref{8.4} to compute the sum in $b$, then \eqref{8.5} to compute the sum in $y$.

Multiplying the two sums completes the proof of \eqref{8.28}.

(iii) By \eqref{8.19}, $\Xi_{i,j,k,l}$ is only non-zero if $i=k$ and $j=l$, so the left-hand side of \eqref{8.29} becomes just a sum over $l$,
\[
\sum_{l\ge 0} \Xi_{k,l,k,l} \odot G_{I-k,J-l,K-k,L-l}
\]
By \eqref{8.18}, $\Omega_{k,\alpha,\beta,\gamma,\delta}$ has $\binom{k}{\alpha+\delta}$ as one of the factors, so it can only be non-zero if $\alpha+\delta\le k$. Coupled with $\alpha\ge \gamma+k$ and $\gamma,\delta\ge 0$, this gives $\alpha=k$, $\gamma=\delta=0$, so the right-hand side of \eqref{8.29} becomes
\[
\sum_{\beta\ge 0} \Omega_{k,k,\beta,0,0} \odot G_{I-k,J-\beta,K-k,L-\beta}
\]
The proof is completed by $\Xi_{k,\beta,k,\beta} = \binom{k+\beta-1}{k-1} = \Omega_{k,k,\beta,0,0}$.

(iv) If $k\le 0$, then $G_{i,j,k,l}=k g_{i,j,k,l} = 0$ by definition. For $K-k\le 0$, analogously $G_{I-i,J-j,K-k,L-l}=0$. For $0<k<K$, we prove \eqref{8.26} by complete induction on $I+J$.

Both sides are $0$ if $I+J<0$, which provides the basis of induction. Assume that \eqref{8.26} holds when $I+J<M$. For $I+J=M$, start from
\begin{equation} \label{8.30}
\sum_{i,j,l\ge 0} H_{i,j,k,l} \odot H_{I-i,J-j,K-k,L-l} 
\end{equation}
and use \eqref{8.24} to replace $H_{i,j,k,l}$ and $H_{I-i,J-j,K-k,L-l}$. That gives four sums, one of terms of the form $\Xi \odot \Xi$, two of the form $\Xi \odot \Omega \odot G$ and one of the form $\Omega\odot \Omega \odot G \odot G$. Use \eqref{8.27} to compute the sum of $\Xi \odot \Xi$, use \eqref{8.29} to replace the sums of $\Xi \odot \Omega \odot G$ by sums of $\Omega \odot \Omega \odot G$, and use the inductive assumption to replace the sum of $\Omega \odot \Omega \odot G \odot G$ by a sum of $\Omega \odot \Omega \odot G$ (this will be possible for all terms except $\Omega_{K-k,0,0,0,0} \odot \Omega_{k,0,0,0,0} \odot G_{I,J,K,L}$ because for that term $I+J$ is not less than $M$).  Finally using \eqref{8.28} to replace the sum of $\Omega\odot\Omega\odot G$ by a sum of $\Omega \odot G$ and using \eqref{8.25} to combine terms, we conclude that \eqref{8.30} is equal to
\begin{equation} \label{8.31}
H_{I,J,K,L} - G_{I,J,K,L} + \sum_{i,j,l\ge 0} G_{i,j,k,l} \odot G_{I-i,J-j,K-k,L-l} 
\end{equation}
However, applying \eqref{8.25} to $H_{I,J,K,L}$, $H_{i,j,k,l}$, $H_{I-i,J-j,K-k,L-l}$, one gets
\begin{equation}\label{8.32}
\frac{\sum_{i,j,l\ge 0} H_{i,j,k,l} \odot H_{I-i,J-j,K-k,L-l}}{\sum_{i,j,l\ge 0} G_{i,j,k,l} \odot G_{I-i,J-j,K-k,L-l}}= \frac{H_{I,J,K,L}} {G_{I,J,K,L}}
\end{equation}
From \eqref{8.30}=\eqref{8.31} and \eqref{8.32}, we conclude that \eqref{8.26} holds for our choice of $I,J,K,L$, which completes the inductive step.

We prove (v) and (vi) simultaneously by induction on $I+J$.

If (vi) holds for $I+J<M$: by \eqref{8.13}, singularities of $f_{I,J,K,L}$ come from a $g_{i,j,k,l}$ with $i+j<I+J$, so (v) then holds for $I+J\le M$.

If (v) holds for $I+J<M$: by applying \eqref{8.26} $K-1$ times, $g_{I,J,K,L}$ can be written as a sum of $K$-fold products of $g_{i,j,1,l}$ with $i+j\le I+J$. Thus, all its non-removable singularities are singularities of a $g_{i,j,1,l}$ with $i+j\le I+J$. By \eqref{8.10}, those can only be of the form \eqref{8.1} with $k=1$, or coming from $f_{i,j,1,l}$. Thus, (vi) holds for $I+J<M$.

For (vii) and (viii), note that  in the $L=0$ case \eqref{8.13} becomes
\begin{equation}\label{8.33}
f_{I,J,K,0} = \xi_{I,J,K,0} + \sum_{\substack{\alpha,\gamma\ge 0 \\ \alpha+\gamma\ge 1}}  \omega_{K,\alpha,0,\gamma,0} \odot g_{I-\alpha,J-\gamma,K+\gamma-\alpha,0}
\end{equation}
where $\xi_{I,J,K,0} = \delta_{I-K} \delta_J$. Induction on \eqref{8.33} using \eqref{8.10} then shows that $f_{I,J,K,0}=g_{I,J,K,0} = 0$ unless $I-J=K$. With this observation in mind, the proof of (vii) and (viii) is analogous to the proof of (v) and (vi) above, using \eqref{8.33} instead of \eqref{8.13}.
\end{proof}

For OPRL, if we are in the $\ell^p$ case, we encounter functions $f_{I,J,K,L}$ and $g_{I,J,K,L}$ with $I+J<p$. Lemma~\ref{L8.1}(\ref{L8.1(v)}),(\ref{L8.1(vi)}) implies that all of their non-removable singularities are of the form \eqref{8.1} with $k=1$ and $i+j<p$. All such points are in the set $A_p$ given by \eqref{5.8}, so all iterations of Lemma~\ref{L6.1} can be performed away from $A_p$.

For OPUC, since $c=0$, terms with $L>0$ vanish. For terms with $L=0$, Lemma~\ref{L8.1}(\ref{L8.1(vii)}),(\ref{L8.1(viii)}) implies that all non-removable singularities of  $f_{I,J,K,0}$ and $g_{I,J,K,0}$ are of the form \eqref{8.1} with $k=i-j=1$ and $i+j<p$. All such points are in the set $A_p$ given by \eqref{5.8}, so  all iterations of Lemma~\ref{L6.1} can be performed away from $A_p$.

\section{Proof in the General Case}

In this section, we complete the proofs of Theorems~\ref{T1.1} and \ref{T1.2} in the general $\ell^p$ case. As hinted before, the key idea will be to use Lemma~\ref{L3.1}(\ref{L3.1(ii)}) and Lemma~\ref{L4.1}(\ref{L4.1(ii)}); we will be able to prove that if $\log r_n$ didn't converge as desired, it would be possible to construct a set of recursion coefficients (corresponding to a measure) for which it diverged uniformly on an interval, contradicting Lemma~\ref{L3.1}(\ref{L3.1(ii)}) or Lemma~\ref{L4.1}(\ref{L4.1(ii)}).

As explained in the previous section, the first step in the proof is to start with \eqref{8.6} and iteratively apply Lemma~\ref{L6.1} to terms of the form
 \begin{align*}
& f_{I,J,K,L}\bigl(\eta;\{\phi_{k_i}\}_{i=1}^I;\{\phi_{l_j}\}_{j=1}^J\bigr) \; \prod_{i=1}^I (h_{k_i}(\eta) \beta_n^{(k_i)}) \; \prod_{j=1}^J (\bar h_{l_j}(\eta) \bar \beta_n^{(l_j)}) \\
& \qquad\qquad\qquad\qquad\qquad\qquad \qquad\qquad \qquad\qquad\quad  \times e^{iK[(n+1)\eta+2\theta_n]} \; c^L
\end{align*}
in increasing order of $I+J$. In the previous section, we have seen that the only singularities we will encounter in these iterations are in $A_p$.

Lemma~\ref{L6.1} can be applied to a term unless  $K=0$ and $\phi \in 2\pi \mathbb{Z}$, so after the iterative procedure, what remains is a sum of such terms,
\begin{align}
 \log\frac{r_{n+1}}{r_n} &  \rel{A_p}  \Re  \sum \Bigl( f_{I,J,0,L}\bigl(\eta;\{\phi_{k_i}\}_{i=1}^I;\{\phi_{l_j}\}_{j=1}^J\bigr) \nonumber \\
 &  \qquad\qquad\qquad \times \prod_{i=1}^I (h_{k_i}(\eta) \beta_n^{(k_i)}) \; \prod_{j=1}^J (\bar h_{l_j}(\eta) \bar \beta_n^{(l_j)}) \;  c^L \Bigr) \label{9.1}
  \end{align}
with the sum going over $(I+J)$-tuples $(k_1,\dotsc,k_I,l_1,\dotsc, l_J)$ with
\begin{equation}\label{9.2}
\phi_{k_1} + \dotsb + \phi_{k_I}  - \phi_{l_1} - \dotsb - \phi_{l_J} = 0
\end{equation}
and $I+J < p$.

At this point, a change of notation will be useful. Our proof in this section will rely on constructing counterexamples, and for that it would be useful to be able to construct $\beta^{(l)}$'s independently. For OPUC this is true, but for OPRL, by Lemma~\ref{L2.2}(\ref{L2.2(vii)}), $\beta^{(l)}$'s come in complex-conjugate pairs: for every $\beta^{(l)}$ there is a $\beta^{(k)} = \bar \beta^{(l)}$. For each such pair, let us keep only one of the two sequences, say $\beta^{(l)}$, and replace $\beta^{(k)}$ everywhere by $\bar \beta^{(l)}$. This is equivalent to replacing \eqref{5.5} by
\begin{equation}\label{9.3}
\alpha_n(\eta) = \sum_{l=1}^{L'} h_l(\eta) (\beta_n^{(l)} + c \bar \beta_n^{(l)})
\end{equation}

Notice that the right-hand side of \eqref{9.1} is the real part of a polynomial in $\beta_n^{(l)}$ and $\bar \beta_n^{(l)}$, with coefficients continuous in $\eta$. Denoting this polynomial by $Q$, \eqref{9.1} becomes
\begin{equation}\label{9.4}
 \log\frac{r_{n+1}}{r_n}   \rel{A_p}  \Re Q(\eta; \beta_n^{(1)}, \dots, \beta_n^{(L)};\bar \beta_n^{(1)}, \dots,\bar \beta_n^{(L)})
\end{equation}
We now make the claim that the right-hand side vanishes identically.

\begin{lemma}\label{L9.1}
For all $\eta \notin A_p + 2\pi\mathbb{Z}$ and all $z_1,\dots, z_L \in \mathbb{C}$,
\begin{equation}\label{9.5}
\Re Q(\eta; z_1,\dots, z_L; \bar z_1,\dots, \bar z_L) = 0
\end{equation}
\end{lemma}

\begin{proof} The proof will proceed by contradiction. Split $Q$ into a sum of homogeneous polynomials $Q_1,\dots,Q_{p-1}$ with $\deg Q_k = k$. If the claim of the lemma is false, then there exists a smallest $k$ such that $\Re Q_k$ does not vanish identically, and a choice of $\eta_0, z_1,\dots, z_L$ such that 
\[
\Re Q_k(\eta_0; z_1,\dots, z_L; \bar z_1,\dots, \bar z_L) \neq 0
\]

Since $Q$ depends only on the values of $p$, the phases $\phi_1, \dots, \phi_L$, and $h_1(\eta),\dots,h_L(\eta)$, but not on $\beta_n^{(l)}$, we are free to make a choice for $\beta_n^{(l)}$. Let
\begin{equation} \label{9.6}
\beta_n^{(l)} =\begin{cases} z_l e^{-i n \phi_l}  n^{-1/(p-1)} & \text{for }n\ge n_0 \\
0 & \text{for }n<n_0
\end{cases}
\end{equation}
Note that $\beta^{(l)} \in \ell^p \cap \GBV{\phi_l}$. Through \eqref{9.3}, this choice of $\beta^{(l)}$ corresponds to a sequence of recursion coefficients, if we choose $n_0$ large enough that the recursion coefficients are in the allowed range ($\lvert \alpha_n \rvert < 1$ for OPUC, $a_n^2 - 1 > -1$ for OPRL). Verblunsky's or Favard's theorem then imply that \eqref{9.6} corresponds to a probability measure on the unit circle or real line. Thus, \eqref{9.4} holds for the choice \eqref{9.6}.

 For every monomial $\beta_n^{(k_1)} \cdots \beta_n^{(k_I)} \bar \beta_n^{(l_1)}\cdots \bar \beta_n^{(l_J)}$ in $Q$, the condition \eqref{9.2} is satisfied, so the factors $e^{-i n \phi_l}$ cancel out completely in $Q$, and  substituting \eqref{9.6} into \eqref{9.4} gives
\begin{equation}\label{9.7}
 \log\frac{r_{n+1}}{r_n}   \rel{A_p}  \sum_{l=1}^{p-1} \Re  Q_l(\eta; z_1, \dots, z_L;\bar z_1, \dots,\bar z_L)  \; n ^{-l / (p-1)}
\end{equation}
Summing \eqref{9.7} in $n$, the non-zero term with $l=k$ will dominate the sum, and since $\sum_{n=1}^\infty n^{-k/(p-1)} = \infty$, this will imply that $\log r_n$ converges  to $+\infty$ or $-\infty$ (depending on the sign of $\Re Q_k$) uniformly on $\eta$ in a neighborhood of $\eta_0$. By Lemma~\ref{L3.1}(\ref{L3.1(ii)}) or Lemma~\ref{L4.1}(\ref{L4.1(ii)}), this is a contradiction, so \eqref{9.5} holds. \end{proof}

Having proved Lemma~\ref{L9.1}, \eqref{9.4} becomes \eqref{5.9}. By Lemma~\ref{L3.1}(\ref{L3.1(i)}) and  Lemma~\ref{L4.1}(\ref{L4.1(i)}), this completes the proof of Theorems \ref{T1.1} and \ref{T1.2}.

\bigskip

\end{document}